\documentclass[a4paper,11pt]{amsart}
\usepackage{amssymb,amsfonts,amscd,amssymb}
\usepackage[all,cmtip, arc]{xy}
  
\usepackage{enumerate}
\usepackage{bbm}
\usepackage{amsmath}
\usepackage{amsmath,amsthm,amssymb,fancyhdr,graphicx,bbm,cancel,color,mathrsfs, yfonts}
\usepackage{xypic}
\usepackage{mathrsfs}
\usepackage{xcolor}

\usepackage{comment}
\usepackage[colorlinks,linkcolor=black,citecolor=black]{hyperref}

\definecolor{shadecolor}{gray}{0.875}

\newtheorem{thm}{Theorem}[section]
\newtheorem{thmx}{Theorem}

\newtheorem{cor}[thm]{Corollary}
\newtheorem{prop}[thm]{Proposition}

\newtheorem{lem}[thm]{Lemma}

\newtheorem{quest}[thm]{Question}

\theoremstyle{definition}
\newtheorem{defn}[thm]{Definition}

\newtheorem{exmp}[thm]{Example}

\newtheorem{rmk}[thm]{Remark}
\newtheorem*{exer*}{Exercise}

\theoremstyle{remark}

\newcommand{\p}{\partial}

\newcommand{\R}{\mathbb R}

\newcommand{\e}{\varepsilon}

\providecommand{\abs}[1]{\left\lvert#1\right\rvert}

\newcommand{\vol}{\mathrm{vol}}
\newcommand{\Amp}{\mathrm{Amp}}

\DeclareMathOperator{\mult}{mult}

\DeclareMathOperator{\Eff}{\overline{Eff}}
\DeclareMathOperator{\Nef}{Nef}
\DeclareMathOperator{\Mov}{Mov}

\DeclareMathOperator{\Sing}{Sing}

\DeclareMathOperator{\codim}{codim}

\DeclareMathOperator{\HConc}{HConc}

\DeclareMathOperator{\sing}{sing}

\makeatletter
\let\c@equation\c@thm
\makeatother
\numberwithin{equation}{section}

\title{Polar transform and local positivity for curves}
\author{Nicholas McCleerey and Jian Xiao}
\date{}

\begin{document}
\maketitle

\begin{abstract}
Using the duality of positive cones, we show that applying the polar transform from convex analysis to local positivity invariants for divisors gives interesting and new local positivity invariants for curves. These new invariants have nice properties similar to those for divisors. In particular, this enables us to give a characterization of the divisorial components of the non-K\"ahler locus of a big class.
\end{abstract}

\tableofcontents

\section{Introduction}
Let $X$ be a smooth projective variety of dimension $n$. Let $L$ be a nef line bundle on $X$. One of the most important invariants measuring the local positivity of $L$ at $x$ is the local Seshadri constant introduced by Demailly \cite{demaillySeshadri}. It is defined as follows:
\begin{equation}\label{eq divisor sesh}
  s_x (L) = \inf_{x\in C} \frac{L\cdot C}{\nu(C, x)},
\end{equation}
where $\nu(C, x) = \mult_x (C)$ is the multiplicity of $C$ at $x$, and the infimum is taken over all irreducible curves passing through $x$. The constant $s_x (L)$ measures the local ampleness of $L$ at $x$. It is well known that $s_x (L)$ can also be characterized as:
\begin{equation}\label{def sesh}
  s_x (L) = \sup\{t\geq 0 | \pi^{*}L - tE\ \textrm{is nef}\},
\end{equation}
where $\pi: Y \rightarrow X$ is the blow-up of $X$ at $x$ and $E = \pi^{-1}(x)$ is the exceptional divisor.
Futhermore, one can define the global Seshadri constant of $L$ by setting
\begin{equation*}
  s(L) = \inf_{x\in X} s_x (L).
\end{equation*}
Seshadri's well known criterion for ampleness can then be stated as: $L$ is ample if and only if $s(L)>0$.

Denote the cone of nef $(1,1)$ classes of $X$ by $\Nef^1 (X)$. It is not hard to see that the Seshadri constant can be defined for any class in $\Nef^1 (X)$, not only for classes given by divisors, by using definition \eqref{def sesh}. As a function defined on the cone $\Nef^1 (X)$, $s_x : \Nef^1 (X) \rightarrow \mathbb{R}$, satisfies:
\begin{itemize}
  \item it is upper semicontinuous and homogeneous of degree one;
  \item it is strictly positive in the interior of $\Nef^1 (X)$;
  \item it is 1-concave; that is, for any $L, M \in \Nef^1 (X)$ we have $s_x (L+M) \geq s_x (L) +s_x (M)$.
\end{itemize}
Moreover, by a direct intersection number calculation we have:
\begin{itemize}
  \item for any $L \in \Nef^1 (X)$, $s_x (L) \leq \vol(L)^{1/n}$.
\end{itemize}
The above properties also hold for the function $s : \Nef^1 (X) \rightarrow \mathbb{R}$.

Another important local positivity invariant for divisors is the local Nakayama constant introduced by Lehmann
(see \cite[Definition 5.1]{lehmann13comparing},
and also \cite{Nak04} for related objects).
It is defined for any pseudo-effective $(1,1)$ class $L$:
\begin{equation*}
  n_x (L) := \sup\{t\geq 0 | \pi^{*}L - tE\ \textrm{is pseudo-effective}\},
\end{equation*}
where $\pi: Y \rightarrow X$ again is the blow-up of $X$ at $x$ and $E = \pi^{-1}(x)$ is the exceptional divisor. Note that $n_x (L)$ can be strictly positive even if $L$ is not big, hence the local Nakayama constant is a more sensitive measure for positivity. Similar to the global Seshadri constant, we can associate to $L$ a global Nakayama constant by:
\begin{equation*}
  n(L) = \inf_{x\in X} n_x (L).
\end{equation*}

Denote the cone of pseudo-effective $(1,1)$ classes of $X$ by $\Eff^1 (X)$. As a function defined on the cone $\Eff^1 (X)$, $n_x : \Eff^1 (X) \rightarrow \mathbb{R}$, satisfies:
\begin{itemize}
  \item it is upper semicontinuous and homogeneous of degree one;
  \item it is strictly positive in the interior of $\Eff^1 (X)$;
  \item it is 1-concave.
\end{itemize}
Moreover, by the mass concentration method of \cite{Dem93} and upper semicontinuity of Lelong numbers we have the following estimate (see Proposition \ref{prop lowerbound} below):
\begin{itemize}
  \item for any $L \in \Eff^1 (X)$, $n_x (L) \geq \vol(L)^{1/n}$.
\end{itemize}
The same properties also hold true for the function $n : \Eff^1 (X) \rightarrow \mathbb{R}$.

\subsection{Polar transform}
By the general theory developed in \cite{lehmxiao16convexity}, given a proper closed convex cone $\mathcal{C}\subset V$ of full dimension in a real vector space $V$, let $\HConc_1 (\mathcal{C})$ be the space of real valued functions defined over $\mathcal{C}$ that are upper semicontinuous, homogeneous of degree one,
strictly positive in the interior of $\mathcal{C}$, and 1-concave. Then one can study the polar transform $\mathcal{H}: \HConc_1 (\mathcal{C}) \rightarrow \HConc_1 (\mathcal{C}^*)$:
\begin{align*}
\mathcal{H}f: \mathcal{C}^* \rightarrow \mathbb{R},\
  w^* \mapsto \inf_{v\in \mathcal{C}^\circ} \frac{w^* \cdot v}{f(v)},
\end{align*}
where $f \in \HConc_1 (\mathcal{C})$, and $\mathcal{C}^* \subset V^*$ is the dual of $\mathcal{C}$.
This is a Legendre-Fenchel type transform with a ``coupling'' function given by the logarithm. By \cite[Proposition 3.2]{lehmxiao16convexity}, we know that $\mathcal{H}f \in \HConc_1 (\mathcal{C}^*)$ whenever $f \in \HConc_1 (\mathcal{C})$, and $\mathcal{H}: \HConc_1 (\mathcal{C}) \rightarrow \HConc_1 (\mathcal{C}^*)$ is a duality transform (i.e., $\mathcal{H}$ is an order-reversing involution.).

In this paper, we apply the polar transform to the following two geometric cases:
\begin{enumerate}
  \item $\mathcal{C}=\Nef^1 (X),\ f = s_x$,
  \item $\mathcal{C}=\Eff^1 (X),\ f = n_x$.
\end{enumerate}
We show that the polar transforms of the local Seshadri and Nakayama constants for $(1,1)$-classes measure the local positivity for $(n-1, n-1)$-classes. In fact, we show the dual of $s_x$ behaves analogously to the Nakayama constant, and similarly for the dual of $n_x$.

\begin{rmk}\label{rmk surface dual}
When $X$ is a surface, it is not hard to see that $(\Nef^1 (X), s_x (\cdot))$ and $(\Eff^1 (X), n_x(\cdot))$ are dual to each other. Hence, by either Theorem \ref{thrm nak curve} or \ref{thrm sesha curve} below, we see that the polar transform of $s_x$ is actually $n_x$ in this case.
\end{rmk}

\subsubsection{Nakayama constants for $(n-1, n-1)$-classes}
For the case $(\Nef^1 (X), s_x)$, note that we have the duality of positive cones
\begin{equation*}
\Nef^1 (X)^* = \Eff_1 (X),
\end{equation*}
where $\Eff_1 (X)$ is the cone of pseudo-effective $(n-1, n-1)$-classes. Then the polar transform gives the following invariant.

\begin{defn}
For any $\alpha \in \Eff_1 (X)$, its local Nakayama constant at $x$, $N_x (\alpha)$, is defined to be the polar of $s_x: \Nef^1 (X)\rightarrow \mathbb{R}$, that is,
\begin{equation*}
  N_x (\alpha) := \mathcal{H}s_x (\alpha)=\inf_{L\in \Nef^1 (X)^\circ} \left(\frac{\alpha\cdot L}{s_x (L)}\right).
\end{equation*}
The global Nakayama constant of $\alpha$ is defined to be
\begin{equation*}
  N (\alpha) :=\inf_{x \in X} N_x (\alpha).
\end{equation*}
\end{defn}

\begin{rmk}
By the properties of polar transforms, the function $N_x (\cdot)\in \HConc_1 (\Eff_1 (X))$.
\end{rmk}

\subsubsection{Seshadri constants for $(n-1, n-1)$-classes}
For the case $(\Eff^1 (X), n_x)$, we have the duality of positive cones
\begin{equation*}
\Eff^1 (X)^* = \Mov_1 (X),
\end{equation*}
where $\Mov_1 (X)$ is the cone of movable $(n-1, n-1)$-classes. Applying the polar transform then gives:

\begin{defn}
For any $\alpha \in \Mov_1 (X)$, its local Seshadri constant at $x$, $S_x (\alpha)$, is defined to be the polar of $n_x: \Eff^1 (X)\rightarrow \mathbb{R}$, that is,
\begin{equation*}
  S_x (\alpha) := \mathcal{H}n_x (\alpha)=\inf_{L\in \Eff^1 (X)^\circ} \left(\frac{\alpha\cdot L}{n_x (L)}\right).
\end{equation*}
The global Seshadri constant of $\alpha$ is defined to be
\begin{equation*}
  S (\alpha) :=\inf_{x \in X} S_x (\alpha).
\end{equation*}
\end{defn}

\begin{rmk}
By the properties of polar transforms, the function $S_x (\cdot)\in \HConc_1 (\Mov_1 (X))$.
\end{rmk}

\begin{rmk}
In principle, it is very difficult to get the exact values of $N_x$ and $S_x$ except in the simplest situations. To get the values, one needs to know the explicit structures of positive cones, and furthermore, the values of $s_x$ and $n_x$ for $(1,1)$-classes -- whose computations in general have equal difficulties.
\end{rmk}

\begin{rmk}\label{rmk sesah def}
When $\alpha$ is a curve class, we can take the dualities of cones in the definitions with respect to the cones generated by divisor classes and curve classes. Repeating the proofs of Thrm. \ref{thrm nak curve} and Thrm. \ref{thrm sesha curve} for these curve quantities, we see that they actually agree with $N_x(\alpha)$ and $S_x(\alpha)$ as defined above -- this follows directly from the geometric characterizations and the fact that an $(n-1, n-1)$-class is in the cone generated by curve classes if and only if it is in the intersection of the transcendental cone and the N\'{e}ron-Severi space. In particular, for a movable curve class $\alpha$, it is not hard to see that $S_x$ will then have an equivalent definition:
\begin{equation*}
  S_x (\alpha) = \inf_{x\in D} \frac{\alpha \cdot D}{\mult_x (D)},
\end{equation*}
where the infimum is taken over all irreducible divisors $D$ passing through $x$ -- see Rmk \ref{rmk nak def}.
This is exactly analogous to \eqref{eq divisor sesh}.
\end{rmk}

\begin{rmk}
Even for surfaces, in general the local/global Seshadri/Nakayama functions are not differentiable everywhere in the interior of the respective positive cones (see e.g. \cite[Example 5.5]{kuronya_jetsep}).
\end{rmk}

While we focus on the local positivity at a point, the Seshadri/Nakayama constants for $(1,1)$-classes have various other generalizations, including one to measure positivity along other subvarieties. Thus, by the polar transform, one can obtain corresponding local positivity invariants for $(n-1, n-1)$-classes; see Section \ref{sec further discuss} for more discussions.

\subsection{Main results}

In \cite{xiao15}, the second named author introduced the volume function $\widehat{\vol}(\cdot)$, defined on $\Eff_1 (X)$, as the polar transform of $(\Nef^1 (X), \vol)$. It has many nice properties (see \cite{xiao15}, \cite{lehmxiao16convexity}).
By the upper bound for $s_x$ and cone duality, we get:

\begin{thmx} \label{thrm nak curve}
For any $\alpha \in \Eff_1 (X)$, its local/global Nakayama constant satisfies:
\begin{itemize}
  \item $N_x(\alpha)\geq \widehat{\vol}(\alpha)^{n-1/n}$.
  \item $N_x (\alpha)$ has the following geometric characterization:
\begin{equation*}
  N_x (\alpha) = \sup\{t\geq 0 | \pi^* \alpha + t e\ \textrm{is pseudo-effective}\},
\end{equation*}
where $\pi: Y \rightarrow X$ is the blow-up of $X$ at $x$ and $e := (-E)^{n-1}$.
\end{itemize}

\end{thmx}

\begin{rmk}
The above result is mirror to the inequality $n(L) \geq \vol(L)^{1/n}$ and the characterization $n_x (L) = \sup\{t\geq 0 | \pi^* L - t E\ \textrm{is pseudo-effective}\}$ for $L \in \Eff^1 (X)$.
\end{rmk}

In \cite{xiao15}, the function $\mathfrak{M}(\cdot)$, defined on $\Mov_1 (X)$, was also introduced. It is the polar transform of $(\Eff^1 (X), \vol)$, and is thoroughly studied in \cite{lehmann2016positiivty}.

\begin{thmx}\label{thrm sesha curve}
The local/global Seshadri constant has the following properties:
\begin{itemize}
  \item For any $\alpha \in \Mov_1 (X)$, $S_x(\alpha) \leq \mathfrak{M}(\alpha)^{n-1/n}$.
  \item $S_x$ has the following geometric characterization:
      \begin{equation*}
       S_x (\alpha) = \sup\{t\geq 0 | \pi^* \alpha + t e\ \textrm{is movable}\},
      \end{equation*}
      where $\pi: Y \rightarrow X$ is the blow-up of $X$ at $x$ and $e = (-E)^{n-1}$.
   \item Suppose that $\alpha\in\Mov_1(X)$. Then $S(\alpha)>0$ if and only if $\alpha \in \Mov_1 (X)^\circ$.
\end{itemize}
\end{thmx}

Let $L \in \Nef^1 (X)$. By definition, it is clear that $s_x (L) \leq n_x(L)$ for any $x\in X$. Due to the estimates in Theorem \ref{thrm nak curve} and Theorem \ref{thrm sesha curve}, and since $\widehat{\vol}(\alpha) \geq \mathfrak{M}(\alpha)$ for any $\alpha \in \Mov_1 (X)$, we immediately get:

\begin{prop}
Let $\alpha \in \Mov_1 (X)$. Then $S_x (\alpha)\leq N_x (\alpha)$ for any $x\in X$.
\end{prop}

By the structure theorem in \cite[Theorem 1.10 and Corollary 3.15]{lehmann2016positiivty}, for any $\alpha\in \Mov_1 (X)$ with $\mathfrak{M}(\alpha)>0$, there is a unique big movable $(1,1)$ class $L_{\alpha}$ (the ``$(n-1)$-th root of $\alpha$'') such that $\alpha = \langle L_{\alpha}^{n-1} \rangle$, and $\alpha$ is on the boundary of $\Mov_1 (X)$ if and only if the non-K\"ahler locus $E_{nK}(L_{\alpha})$ has some divisorial component. The following result characterizes the vanishing locus
\begin{equation*}
\mathcal{V}(\alpha):=\{x \in X\,|\, S_x (\alpha) =0\}.
\end{equation*}

\begin{thmx}\label{thrm local sesha curve}
Let $\alpha\in \Mov_1 (X)$ be a class on the boundary of $\Mov_1 (X)$ with $\mathfrak{M}(\alpha)>0$. Then we have:
\begin{equation*}
  S_x (\alpha)=0 \Leftrightarrow x \in\textrm{the divisorial components of}\  E_{nK}(L_{\alpha}).
\end{equation*}
\end{thmx}

\begin{rmk}
Let $L\in \Nef^1 (X)$, then by \cite[Theorem 2.7]{tosatti-seshadri} (which in turn is a simple application of \cite{collinstosatti_kahlercurrent}) we have  that
\begin{equation*}
  s_x (L)=0 \Leftrightarrow x \in E_{nK}(L).
\end{equation*}
Let $L \in \Eff^1 (X)$ be a big class, and let $\alpha = \langle L^{n-1}\rangle$. By the proof of Theorem \ref{thrm local sesha curve} in Section \ref{sec local pos}, we actually have a slight improvement
\begin{equation*}
  S_x (\alpha)=0 \Leftrightarrow x \in\textrm{the divisorial components of}\  E_{nK}(L).
\end{equation*}
Thus, the invariant $S_x (\cdot)$ can detect more information.
\end{rmk}

\begin{rmk}
While our discussions are focused on $(n-1, n-1)$-classes over smooth projective varieties defined over $\mathbb{C}$, the results of course holds true for any curve classes.
\end{rmk}

\subsection{Relation to other work} Independently, M.~Fulger \cite{fulger2017seshadri} has also studied Seshadri constants for movable curve classes, and his work has substantial overlap with the present paper. The starting point of his work is the equivalent definition for $S_x$ in Remark \ref{rmk sesah def}. For movable curve classes, he has obtained the same results as in Theorem \ref{thrm sesha curve}, and a different (slightly weaker, but essentially the same) result as that in Theorem \ref{thrm local sesha curve}. Besides these results, analogous to \cite{demaillySeshadri}, he also gave a ``jet separation'' interpretation for $S_x(\alpha)$ when $\alpha=\{[C]\}$ is an integral class in the interior of $\Mov_1 (X)$. He also discussed various interesting examples and the analogy of $S_x$ for nef dual $(k,k)$ cycle classes.

\subsection{Organization}
In Section \ref{sec prel}, we give a brief introduction to positivity, and recall some background material. Section \ref{sec nak divisor} is devoted to the study of Nakayama constants for $(1,1)$-classes, which will be applied to the study of local positivity for $(n-1, n-1)$-classes. In Section \ref{sec local pos}, we give the proof of the main results and discuss various generalizations of local positivity invariants for $(n-1, n-1)$-classes.

\subsection*{Acknowledgement}
We would like to thank B.~Lehmann and V.~Tosatti for helpful comments on this work. We also would like to thank M.~Fulger for his nice correspondences and sharing us with his preprint. The first named author would also like to express extra gratitude to V.~Tosatti for his continued interest and encouragement on this project. The first named author was also supported in part by the National Science Foundation grant ``RTG: Analysis on Manifolds'' at Northwestern University.

\section{Preliminaries}\label{sec prel}

Throughout the paper, with a few expectations, the capital letter $L$ will denote a $(1,1)$-class and the Greek letter $\alpha$ will denote an $(n-1, n-1)$-class.

\subsection{Positive classes}
Let $X$ be a smooth projective variety of dimension $n$. We will let $H^{1,1}(X, \mathbb{R})$ denote the real de Rham cohomology group of bidegree $(1,1)$. We will be interested in the following cones in $H^{1,1}(X,\mathbb{R})$:
\begin{itemize}
  \item $\Nef^1 (X)$: the cone of nef $(1,1)$-classes, i.e., the closure of the K\"ahler cone;
  \item $\Mov^{1}(X)$: the cone of movable $(1,1)$-classes, i.e., the closed cone generated by classes of the form $\mu_* \widetilde{A}$, where $\mu$ is a modification and the $\widetilde{A}$ is a K\"ahler class upstairs.
  \item $\Eff^1 (X)$: the cone of pseudo-effective $(1,1)$-classes, i.e., the cone generated by classes which contain a positive current.
\end{itemize}
Classes in the interior of $\Eff^{1}(X)$ are known as big classes; these are exactly the classes admitting a K\"ahler current (i.e. strictly positive current), or equivalently, having strictly positive volume. As a consequence of Demailly's regularization theorem \cite{demailly1992regularization}, any big class contains a K\"ahler current with analytic singularities. Let $L \in \Eff^1 (X)^\circ$
be a big class. Its non-K\"ahler locus is defined to be
\begin{equation*}
  E_{nK}(L) = \bigcap_{T_+ \in L} \Sing(T_+),
\end{equation*}
where the intersection is taken over all K\"ahler currents $T_+ \in L$ with analytic singularities, and $\Sing(T_+)$ is the singular set of $T_+$. It is proved in \cite{Bou04} that there is a K\"ahler current $T_L \in L$ with analytic singularities such that $E_{nK}(L) = \Sing(T_L)$. The ample locus $\Amp(L)$ of $L$ is the complement of $E_{nK}(L)$:
\begin{equation*}
\Amp(L) = X \setminus E_{nK}(L).
\end{equation*}

\begin{rmk}
When $L$ is given by a big divisor class, we have $E_{nK}(L) = \mathbb{B}_+ (L)$ -- the augmented base locus of $L$.
\end{rmk}

Let $H^{n-1,n-1}(X, \mathbb{R})$ denote the real de Rham cohomology group of bidegree $(n-1,n-1)$. We will be interested in the following cones in $H^{n-1,n-1}(X,\mathbb{R})$:
\begin{itemize}
\item $\Eff_1 (X)$: the cone of pseudo-effective $(n-1,n-1)$-classes, i.e., the closed cone generated by classes which contain a positive current;
\item $\Mov_{1}(X)$: the cone of movable $(n-1,n-1)$-classes, i.e., the closed cone generated by classes of the form $\mu_* (\widetilde{A}_1 \cdot ... \cdot \widetilde{A}_{n-1})$, where $\mu$ is a modification and the $\widetilde{A}_i$ are K\"ahler classes upstairs.
\end{itemize}

By \cite{DP04}, \cite{BDPP13} (and its extension to the transcendental case \cite{nystrom2016duality}), we have the following dualities of positive cones:
\begin{equation*}
\Nef^1 (X)^* = \Eff_1 (X),\ \ \Eff^1 (X)^* = \Mov_1 (X).
\end{equation*}

\begin{rmk}
The positive cone $\Mov_1 (X)$ can also be defined by positive products, denoted by $\langle-\rangle$, of pseudo-effective $(1,1)$-classes (see e.g. \cite{BDPP13}).
\end{rmk}

\subsection{Volume functions for $(n-1, n-1)$-classes}

In \cite{xiao15}, the polar transform was applied to the volume function on $(1,1)$-classes over both cones $\Eff^1(X)$ and $\Nef^1(X)$, defining two volume type functions for $(n-1, n-1)$-classes. For classes in $\Eff_1 (X)$, we have the polar of $(\Nef^1 (X), \vol)$:
\begin{align*}
  \widehat{\vol}: &\Eff_1 (X) \rightarrow \mathbb{R},\\
   &\alpha \mapsto \widehat{\vol}(\alpha) = \inf_{L \in \Nef^1 (X)^\circ} \left(\frac{\alpha\cdot L}{\vol(L)^{1/n}}\right)^{n/n-1}.
\end{align*}
For classes in $\Mov_1 (X)$, we have the polar of $(\Eff^1 (X), \vol)$:
\begin{align*}
\mathfrak{M}: &\Mov_1 (X) \rightarrow \mathbb{R},\\ &\alpha \mapsto \mathfrak{M}(\alpha) = \inf_{L \in \Eff^1 (X)^\circ} \left(\frac{\alpha\cdot L}{\vol(L)^{1/n}}\right)^{n/n-1}.
\end{align*}

These two volume functions will give the lower/upper bounds for the Seshadri and Nakayama functions of $(n-1, n-1)$-classes. Note that we always have $\mathfrak{M}(\alpha)\leq\widehat{\vol}(\alpha)$ for $\alpha\in\Mov_1(X)$.

\subsection{Lelong numbers}
Recall that the Lelong number of a closed positive $(1,1)$-current $T$ at a point $x\in X$ is defined to be:
\begin{equation*}
\nu(T,x) := \lim_{r\rightarrow 0^+} \frac{1}{r^{2(n-1)}}\int_{B_r(x)} T\wedge(dd^c\abs{z}^2)^{n-1},
\end{equation*}
where $(U,z)$ is a local coordinate chart around $x$ and $n$ is the dimension of $X$. It is shown in \cite[Chapter 3]{Dem_AGbook} that $\nu(T,x)$ is independent of coordinate chart, and that the numbers
\begin{equation*}
\nu(T,x,r) := \frac{1}{r^{2(n-1)}}\int_{B_r(x)} T\wedge(dd^c\abs{z}^2)^{n-1}
\end{equation*}
decrease to $\nu(T,x)$ as $r$ tends to zero. It is also shown that Lelong numbers are upper semicontinuous in both arguments:

\begin{lem}
\label{lem lelong semicontinuous}
The function $x\mapsto \nu(T, x)$ is upper semicontinuous on $x$;
The function $T\mapsto \nu(T, x)$ is also upper semicontinuous on $T$.
\end{lem}

We shall need a slight strengthening of the above result:

\begin{lem}\label{lem lelong semicontinuous two}
Suppose that $x_k\rightarrow x$, and that $T_k\geq 0$ are positive $(1,1)$-currents that converge weakly to a positive current $T$. Then
\begin{equation*}
\nu(T,x)\geq \limsup_{k\rightarrow\infty} \nu(T_k,x_k).
\end{equation*}
\end{lem}
\begin{proof}
 Fix a coordinate chart $U$ around $x$ such that for all $k$ large enough $x_k \in U$. By the definition of Lelong numbers, we have
\begin{equation*}
\nu(T_k,x_k)= \lim_{r\rightarrow 0} \nu(T_k,x_k,r).
\end{equation*}
Let $v:= \limsup_{k\rightarrow\infty}\nu(T_k,x_k)$, and fix a subsequence (which we shall again call $T_k$) so that
\begin{equation*}
\nu(T_k,x_k)\rightarrow v.
\end{equation*}
Note that the positive measures $ T_k\wedge(dd^c |z|^2)^{n-1}$ converge weakly to the positive measure $T\wedge(dd^c |z|^2)^{n-1}$. Fix an $r > 0$ sufficiently small and let $\varepsilon > 0$ be arbitrary. Then $|x - x_k| <\varepsilon$ and $\nu(T_k,x_k)\geq v-\varepsilon$ for all $k$ sufficiently large, so we have:
\begin{equation*}
v - \varepsilon \leq \nu(T_k,x_k)\leq \nu(T_k,x_k,r)\leq \left(\frac{r+\varepsilon}{r}\right)^{2(n-1)}
\nu(T_k,x,r+\varepsilon).
\end{equation*}
First letting $k\rightarrow \infty$ and then letting $\e\rightarrow 0$ imply $v\leq \nu(T,x,r)$, and so we conclude by letting $r\rightarrow 0$.
\end{proof}

\section{Nakayama constant for $(1, 1)$-classes}\label{sec nak divisor}

In this section, we study some basic properties of the Nakayama constants for $(1, 1)$-classes, which will prove useful in the study of local positivity for $(n-1, n-1)$-classes.
Recall that, given $L\in \Eff^1 (X)$, the local Nakayama constant $n_x (L)$ is defined by
\begin{equation*}
  n_x (L) = \sup\{t\geq 0\,|\,\pi^* L - t E \in \Eff^1 (Y)\},
\end{equation*}
where $\pi: Y \rightarrow X$ is the blow-up of $X$ at $x$ and $E$ is the exceptional divisor.
Firstly, we note that $n_x (L)$ can be calculated by Lelong numbers.

\begin{prop} \label{prop nak lelong}
Let $L\in \Eff^1 (X)$ and $x\in X$. Then we have
\begin{equation*}
n_x(L) =  \sup_{0 \leq T \in L}\nu(T,x),
\end{equation*}
where $\nu(T, x)$ is the Lelong number at $x$, and the supremum is taken over all positive currents $T$ in $L$.
\end{prop}

\begin{proof}
Let $\pi: Y \rightarrow X$ be the blow-up of $X$ at $x$ and let $E$ be the exceptional divisor. Then if $T\in L$ is a positive current, we have that
$\nu(\pi^*T, E) = \nu(T,x)$,
where $\nu(\pi^*T, E)$ is the generic Lelong number along $E$, defined as $\nu(\pi^*T, E) := \inf_{x\in E} \nu(\pi^*T, x)$.

By Siu's decomposition theorem \cite{siu1974analyticity}, there is a positive current $S$ so that
$\pi^*T = S + \nu(T,x)[E]$. Thus, the class $\pi^*L - \nu(T, x)[E]$ is pseudo-effective, implying
\begin{equation*}
n_x(L) \geq  \sup_{0 \leq T \in L}\nu(T,x).
\end{equation*}

For the other direction, suppose that $\pi^*L-t[E]$ is pseudo-effective. Then it contains a closed positive current $S$, and $S + t[E]$ is then a closed positive current in $\pi^*L$. By \cite[Proposition 1.2.7]{Bou02}, there is a unique positive current $T\in L$ such that $\pi^*T = S + t[E]$. In particular, this implies
\begin{equation*}
n_x(L) \leq  \sup_{0 \leq T \in L}\nu(T,x).
\end{equation*}
This finishes the proof of the equality.
\end{proof}

\begin{rmk}\label{rmk nak def}
If $L$ is a big line bundle, then
\begin{equation*}
n_x (L) = \sup\{\mult_x (D)| D\in L\},
\end{equation*}
where $D$ is taken over all rational effective divisors in $L$. This is a direct consequence of \cite[Corollary 14.23]{demaillyAGbook}.
\end{rmk}

By the definition of Lelong numbers, $n_x (L)$ is always bounded above.

\begin{prop}\label{prop nak upper}
Let $A$ be fixed a K\"ahler class on $X$. Then there is a constant $c>0$ (independent of $x$) such that for all $L\in\Eff^1(X)$:
\begin{equation*}
n_x (L) \leq c (L\cdot A^{n-1}).
\end{equation*}
\end{prop}

\begin{proof}
This follows directly from the definition of Lelong numbers.
\end{proof}

\begin{exmp}
Suppose $E$ is a rigid effective divisor, in the sense that its cohomology class has only one positive current, namely the integration current $[E]$. Then $n_x(E) = \mathrm{mult}_x(E)$. This shows that $n_x(L)$ need not be zero for $L\in\partial \Eff^1 (X)$.
\end{exmp}

\begin{exmp}
Suppose that $L$ is a semi-ample line bundle. Consider the Iitaka fibration for $L$, $\varphi_L: X\rightarrow Y$, and suppose further that $Y$ is smooth. Then $L = \varphi_L^*A$, for some K\"ahler class $A$ on $Y$. It is well known that for any $y \in Y$ there exists a positive current $T_y\in A$ with an isolated analytic singularity at $y$, hence $\nu(T_y, y) > 0$. Then by the result on pull-backs mentioned in Proposition \ref{prop nak lelong}, we have that
\begin{equation*}
\nu(\varphi_L ^*T_y, \varphi_L ^{-1}(y)) > 0.
\end{equation*}
In particular, $n_x(L) > 0$ for all $x\in \varphi_L ^{-1}(y)$. Since $\varphi_L$ is surjective, we have that $n_x(L) > 0$ for all $x\in X$. 
\end{exmp}

Let $L = P(L) + [N(L)]$ be the divisorial Zariski decomposition of $L$ (see e.g \cite{Bou04}). Then for any positive current $T\in L$, we have
\begin{equation*}
T = S + N(L)
\end{equation*}
for some positive current $S \in P(L)$. In particular,
\begin{equation*}
n_x(L) = n_x(P(L)) + n_x([N(L)]),
\end{equation*}
which shows that $n_x$ is not strictly log-concave in $\Eff^1 (X)^\circ$.
This also shows that $n_x(L) = n_x(P(L))$ for all $x\in\mathrm{Amp}(L)$.

Next we show that $n_x (L)$ is achieved by some positive current $T_x\in L$ -- it is a direct consequence of  Lemma \ref{lem lelong semicontinuous}.

\begin{prop} \label{prop maximal lelong}
Let $L \in \Eff^1 (X)$. For any $x \in X$, there is a positive current $T_x \in L$ such that
\begin{equation*}
  n_x (L) = \nu(T_x, x).
\end{equation*}
\end{prop}

\begin{proof}
Take a sequence of positive currents $T_k \in L$ such that
$\lim_{k\rightarrow \infty} \nu(T_k, x) = n_x (L)$.
Since the currents $T_k$ are in the same class, their masses are bounded above by a fixed constant, implying the sequence is compact in the weak topology of currents. Thus, after taking a subsequence and relabeling, we can assume that $T_{k}$ converges to a current $ T_x \in L$. By Lemma
\ref{lem lelong semicontinuous} and the definition of $n_x (L)$ we immediately see
$n_x (L) = \nu(T_x, x)$.

\end{proof}

\begin{cor}\label{cor semicont zeta}
The function $x\mapsto n_x (L)$ is upper semicontinuous.
\end{cor}

\begin{proof}
This is now immediate by Proposition \ref{prop maximal lelong} and Lemma \ref{lem lelong semicontinuous two}.
\end{proof}

\begin{rmk}\label{rmk semicont n_x}
In \cite{fulger2017seshadri}, Fulger shows that if $\alpha$ is a movable curve class, then $x\mapsto S_x(\alpha)$ is lower semicontinuous in the countable Zariski topology. Using this and the definition of $S_x$ as the polar transform of $n_x$, one easily sees that if $L$ is a divisor class, then $x\mapsto n_x(L)$ is actually upper semicontinuous in this same topology, an improvement over Cor. \ref{cor semicont zeta} in this case.
\end{rmk}

\begin{prop}\label{prop lowerbound}
We have that $n_x(L)\geq\vol(L)^{1/n}$ for all pseudo-effective $(1,1)$-classes $L$.
\end{prop}

\begin{proof}
This is a direct consequence of Corollary \ref{cor semicont zeta} and a result
in \cite[Section 6]{Dem93}.

First assume that $L$ is big and nef. In \cite{Dem93}, by solving a family of complex Monge-Amp\`{e}re equations, for any $x\in X$ and positive number $\tau$ satisfying $\tau^n < L^n$, there is a positive current $T\in L$ such that $\nu(T, x)\geq \tau$. In particular, taking a limit of such currents as $\tau\rightarrow \vol(L)^{1/n}$ gives a positive current in $L$ with Lelong number at $x$ at least $\vol(L)^{1/n}$. This shows that $n_x (L) \geq \vol(L)^{1/n}$.

In the case when $L$ is big, we first show that for any $x\in \Amp(L)$, $n_x (L) \geq \vol(L)^{1/n}$. There are two slightly different ways to achieve this desired lower bound. The first method is to apply the result of \cite{BEGZ10MAbig} and the same argument of \cite{Dem93}. The second method is to apply Fujita approximation for $L$; then for any $x \in \Amp(\alpha)$, we reduce to the K\"ahler case upstairs.
Thus for any $x \in \Amp(L)$, we see that $n_x (L) \geq \vol(L)^{1/n}$.

For any point $x \in X\setminus \Amp(L)$, take a sequence of points $x_k \in \Amp(L)$ limiting to $x$ -- applying Corollary \ref{cor semicont zeta} and taking the limit of $n_{x_k} (L)$ as $x_k$ tends to $x$ gives the lower bound.

\end{proof}

\section{Local positivity for $(n-1, n-1)$-classes}\label{sec local pos}
In this section, we give the proof of the main results and discuss various generalizations of the local positivity invariants for $(n-1, n-1)$-classes.

\subsection{Nakayama constants for pseudo-effective $(n-1, n-1)$-classes}

\subsubsection{Proof of Theorem \ref{thrm nak curve}}
Recall that we need to prove the following result:
\begin{thm}[Theorem \ref{thrm nak curve}] 
For any $\alpha \in \Eff_1 (X)$, its local/global Nakayama constant satisfies:
\begin{itemize}
  \item $N_x(\alpha)\geq \widehat{\vol}(\alpha)^{n-1/n}$.
  \item $N_x (\alpha)$ has the following geometric characterization:
\begin{equation*}
  N_x (\alpha) = \sup\{t\geq 0 | \pi^* \alpha + t e\ \textrm{is pseudo-effective}\},
\end{equation*}
where $\pi: Y \rightarrow X$ is the blow-up of $X$ at $x$ and $e := (-E)^{n-1}$.
\end{itemize}
\end{thm}

\begin{proof}
Recall that, given $\alpha \in \Eff_1 (X)$, its local Nakayama constant $N_x (\alpha)$ at $x$ is defined by
\begin{equation*}
  N_x (\alpha) = \inf_{L \in \Nef^1 (X)^\circ} \left(\frac{\alpha\cdot L}{s_x (L)}\right).
\end{equation*}
Since $s_x (L) \leq \vol(L)^{1/n}$, we get
\begin{equation*}
  N_x (\alpha) \geq \inf_{L \in \Nef^1 (X)^\circ} \left(\frac{\alpha\cdot L}{\vol(L)^{1/n}}\right) \geq \widehat{\vol}(\alpha)^{n-1/n}.
\end{equation*}
In particular, we get $N(\alpha) \geq \widehat{\vol}(\alpha)^{n-1/n}$ as desired.

It remains to give the geometric characterization of $N_x (\alpha)$. We need to verify that
\begin{equation*}
  N_x (\alpha) = \sup\{t\geq 0 | \pi^* \alpha + te \in \Eff_1 (Y)\},
\end{equation*}
where $\pi: Y \rightarrow X$ is the blow-up of $X$ at $x$ and $e = (-E)^{n-1}$.

As $\pi$ is the blow-up at one point, it is easy to see that $\pi^*\alpha\in\Eff_1(Y)$ whenever $\alpha\in\Eff_1(X)$ -- in general, this is not true for blow-up along a higher dimensional subvariety. To see this, we just need to note that the class of a K\"ahler current with isolated singularities must be a K\"ahler class. In particular, for any nef class $L' \in \Nef^1 (Y)$, we have $\pi_* L' \in \Nef^1 (X)$, thus
\begin{equation*}
  \pi^* \alpha \cdot L' = \alpha \cdot \pi_* L' \geq 0
\end{equation*}
whenever $\alpha \in \Eff_1 (X)$. Then by the cone duality $\Nef^1 (Y)^* = \Eff_1 (Y)$, we have $\pi^*\alpha\in\Eff_1(Y)$.

Suppose then that $L' \in \Nef^1 (Y)$. Since $Y$ is the blow up of $X$ at a point,
\begin{equation*}
  L' = \pi^* L - s E
\end{equation*}
for some $L\in H^{1,1}(X, \mathbb{R}), s\in \mathbb{R}$. It is not hard to see that $L = \pi_* L' \in \Nef^1 (X)$ and $s\geq 0$. Now, by the duality $\Nef^1 (Y) ^* = \Eff_1 (Y)$, $\pi^* \alpha + te \in \Eff_1 (Y)$ if and only if
\begin{equation*}
  (\pi^* \alpha + te) \cdot (\pi^* L -s E) \geq 0,
\end{equation*}
which is equivalent to the inequality
\begin{equation*}
  t \leq \frac{\alpha\cdot L}{s}
\end{equation*}
for all $L$ nef and all $s\geq 0$ satisfying $\pi^* L - s E \in \Nef^1 (Y)$. Thus $\pi^* \alpha + te \in \Eff_1 (Y)$ if and only if
\begin{equation*}
  t \leq \frac{\alpha\cdot L}{s_x (L)}
\end{equation*}
for all $L$ nef, implying that
\begin{equation*}
N_x (\alpha) = \sup\{t\geq 0 | \pi^* \alpha +te \in \Eff_1 (Y)\}.
\end{equation*}
This finishes the proof of Theorem \ref{thrm nak curve}.

\end{proof}

\begin{rmk}
It is easy to see that Theorem \ref{thrm nak curve} also holds over non-projective K\"ahler manifolds.
\end{rmk}

\begin{rmk}
By the lower semicontinuity of the Seshadri function $x\mapsto s_x (L)$ (see \cite{LazarPositivity-1}), it is clear that the function $x\mapsto N_x (\alpha)$ is upper semicontinuous when $\alpha$ is a curve class and $X$ is endowed with the countable Zariski topology.
\end{rmk}

\subsection{Seshadri constants for movable $(n-1, n-1)$-classes}

\subsubsection{Proof of Theorem \ref{thrm sesha curve}} Recall that we need to prove the following result:

\begin{thm}[Theorem \ref{thrm sesha curve}]
The local/global Seshadri constant has the following properties:
\begin{itemize}
  \item For any $\alpha \in \Mov_1 (X)$, $S_x(\alpha) \leq \mathfrak{M}(\alpha)^{n-1/n}$.
  \item $S_x$ has the following geometric characterization:
      \begin{equation*}
       S_x (\alpha) = \sup\{t\geq 0 | \pi^* \alpha + t e\ \textrm{is movable}\},
      \end{equation*}
      where $\pi: Y \rightarrow X$ is the blow-up of $X$ at $x$ and $e = (-E)^{n-1}$.
   \item Suppose that $\alpha\in\Mov_1(X)$. Then $S(\alpha)>0$ if and only if $\alpha \in \Mov_1 (X)^\circ$.
\end{itemize}
\end{thm}

\begin{proof} 

Recall that, given a movable class $\alpha \in \Mov_1 (X)$, its local Seshadri constant $S_x (\alpha)$ at $x$ is defined by
\begin{equation*}
  S_x (\alpha) = \inf_{L \in \Eff^1 (X)^\circ}\left(\frac{\alpha\cdot L}{n_x (L)}\right).
\end{equation*}
By Proposition \ref{prop lowerbound}, $n_x (L) \geq \vol(L)^{1/n}$, so
\begin{equation*}
  S_x (\alpha) \leq \inf_{L \in \Eff^1 (X)^\circ}\left(\frac{\alpha\cdot L}{\vol(L)^{1/n}}\right) = \mathfrak{M}(\alpha)^{n-1/n}.
\end{equation*}

Similar to the geometric characterization of the local Nakayama function $N_x (\cdot)$, using the duality $\Eff^1 (X)^* =\Mov_1 (X)$ we can show the geometric characterization of $S_x (\cdot)$.

Let $\pi: Y \rightarrow X$ be the blow-up of $X$ at $x$, and let $E = \pi^{-1}(x)$. Note that $\pi^* \alpha \in \Mov_1 (Y)$ whenever $\alpha\in \Mov_1 (X)$, since $\pi_* \Eff^1 (Y) \subset \Eff^1 (X)$.
Let $L' = \pi^*L + sE \in \Eff^1 (Y)$. Then $L \in \Eff^1 (X)$.
Note that $\pi^*\alpha + te \in \Mov_1 (X)$ if and only if
\begin{equation*}
  (\pi^*\alpha + te)\cdot (\pi^*L + sE) \geq 0
\end{equation*}
for all $L \in \Eff^1 (X)$ and all real numbers $s$ such that $\pi^*L + sE \in \Eff^1 (Y)$. This inequality is obvious for $s \geq 0$, so we only need to consider the case $s<0$. Then just as the proof of Theorem \ref{thrm nak curve},
we have that $\pi^* \alpha + te \in \Mov_1 (Y)$ if and only if
\begin{equation*}
  t \leq \frac{\alpha\cdot L}{n_x (L)}
\end{equation*}
for all $L$ pseudo-effective, implying that
\begin{equation*}
  S_x (\alpha) = \sup\{t\geq 0 | \pi^* \alpha +te \in \Mov_1 (Y)\}.
\end{equation*}

Finally, we prove our characterization of the interior of the movable cone -- that if $\alpha\in\Mov_1(X)$, then the global Seshadri constant $S(\alpha)>0$ if and only if $\alpha \in \Mov_1 (X)^\circ$.

Suppose first that $\alpha \in \Mov_1 (X)^\circ$. Then there is a K\"ahler class $A$ such that $\alpha - A^{n-1} \in \Mov_1 (X)$. In particular, for any $L \in \Eff^1(X)\setminus\{0\}$ we have
\begin{equation*}
  \frac{\alpha\cdot L}{A^{n-1}\cdot L} \geq 1.
\end{equation*}
By Proposition \ref{prop nak upper}, we have the upper bound for $n_x(L)$:
\begin{equation*}
  n_x (L) \leq c (L \cdot A^{n-1}).
\end{equation*}
Thus,
\begin{equation*}
  S(\alpha) \geq \frac{\alpha\cdot L}{c (L \cdot A^{n-1})}\geq c^{-1}>0.
\end{equation*}

For the other direction, suppose that $S(\alpha)>0$. Then by the inequality $S(\alpha) \leq \mathfrak{M}(\alpha)^{n-1/n}$, we have that $\mathfrak{M}(\alpha)>0$. By the structure theorem \cite[Theorem 1.10]{lehmann2016positiivty}, $\mathfrak{M}(\alpha)>0$ implies that
\begin{equation*}
  \alpha = \langle L_{\alpha}^{n-1}\rangle
\end{equation*}
for a unique big and movable $(1,1)$ class $L_{\alpha}$. It remains to prove that $\codim E_{nK}(L_{\alpha}) \geq 2$. If not, then $E_{nK}(L_{\alpha})$ contains some divisorial component $D$, and so by \cite{BFJ09} (and its extension to the transcendental situation \cite{nystrom2016duality}), $\alpha \cdot D =0$. Then for $x \in D$ we get
\begin{equation*}
  S(\alpha)\leq S_x (\alpha) \leq \frac{\alpha\cdot D}{\nu(D, x)} =0,
\end{equation*}
forcing $S(\alpha)=0$, a contradiction. This yields $\codim E_{nK}(L_{\alpha}) \geq 2$, and hence $\alpha\in \Mov_1 (X)^\circ$.
\end{proof}

\begin{rmk}
We remark that there is a slightly different way to see $\alpha \in \Mov_1 (X)^\circ$. We need to verify that $\alpha \cdot L >0$ for any non-zero class $L\in \Eff^1 (X)$. 
Since $\mathfrak{M}(\alpha)>0$, by \cite[Lemma 3.10]{lehmann2016positiivty}, for any non-zero $L \in \Mov^1 (X)$ we have $\alpha \cdot L>0$. We also have the Zariski decomposition \cite{Bou04}:
\begin{equation*}
L = P(L)+ N(L),
\end{equation*}
where $N(L)$ is given by an effective divisor.
Then $S(\alpha)>0$ implies $\alpha \cdot N(L)>0$ whenever $N(L)\neq 0$, as then there is a point $x$ so that $n_x(N(L)) > 0$. This finishes the verification. As \cite[Lemma 3.10]{lehmann2016positiivty} holds over a compact K\"ahler manifold which is not projective, this result holds in that generality as well.
\end{rmk}

\begin{rmk}\label{rmk semicont S_x}
Analogous to the semicontinuity of $s_x$ (see \cite{LazarPositivity-1}), Fulger \cite{fulger2017seshadri} shows that if $\alpha$ is a movable curve class, then $x\mapsto S_x(\alpha)$ is lower semicontinuous with respect to the countable Zariski topology.
\end{rmk}

\subsubsection{Proof of Theorem \ref{thrm local sesha curve}} We give a characterization of the vanishing locus of $S_x (\alpha)$.

\begin{thm}[Theorem \ref{thrm local sesha curve}]
Let $\alpha\in \Mov_1 (X)$ be a class on the boundary of $\Mov_1 (X)$ with $\mathfrak{M}(\alpha)>0$. Then we have:
\begin{equation*}
  S_x (\alpha)=0 \Leftrightarrow x \in\textrm{the divisorial components of}\  E_{nK}(L_{\alpha}).
\end{equation*}
\end{thm}

\begin{proof}

Let $\alpha\in \Mov_1 (X)$ be a class on the boundary of $\Mov_1 (X)$ satisfying $\mathfrak{M}(\alpha) >0$. We aim to give a description of the vanishing locus of $S_x (\alpha)$. By \cite[Theorem 1.10 and Corollary 3.15]{lehmann2016positiivty}, we have a unique big and movable $(1,1)$ class $L_{\alpha}$ satisfying $\alpha = \langle L_{\alpha}^{n-1}\rangle$ and $\codim E_{nK} (L_{\alpha}) =1$.

One direction is immediate -- by the argument in the proof of Theorem \ref{thrm sesha curve}, it is clear that if $x$ is in the divisorial components of $E_{nK}(L_{\alpha})$, then $S_x (\alpha)=0$.

For the other direction, let us go back to the definition of $S_x (\alpha)$:
\begin{equation*}
  S_x (\alpha) = \inf_{L \in \Eff^1 (X)^\circ} \frac{\alpha\cdot L}{n_x (L)} = \inf_{L \in \Eff^1 (X)^\circ, n_x (L) =1} \alpha\cdot L.
\end{equation*}
Take a sequence of $L_k$ such that $\alpha\cdot L_k$ tends to $S_x (\alpha) =0$ and $n_x (L_k)=1$. Consider the Zariski decomposition of the sequence:
\begin{equation*}
 L_k = P(L_k) +N(L_k).
\end{equation*}
After taking a subsequence, we have
\begin{equation}\label{eq lim pos neg}
  \alpha\cdot P(L_k)\rightarrow 0\ \textrm{and}\ \alpha\cdot N(L_k)\rightarrow 0.
\end{equation}

By \cite[Theorem 3.10]{lehmann2016positiivty}, since $\mathfrak{M}(\alpha)>0$, we know that $\alpha$ is an interior point of the dual of $\Mov^1 (X)$. Thus $\alpha\cdot P(L_k)\rightarrow 0$ implies that the sequence $P(L_k)$ is compact, and without loss of generality we can assume that $\lim_k P(L_k) =P$. Then
\begin{equation*}
0\leq \alpha \cdot P \leq \lim_k \alpha \cdot L_k = S_x(\alpha) = 0 ,
\end{equation*}
and so $\alpha \cdot P =0$. Since $\mathfrak{M}(\alpha)>0$, by \cite[Theorem 3.10]{lehmann2016positiivty} again, we get $P=0$.
By upper semicontinuity of $n_x (\cdot)$, this implies $\lim_{k\rightarrow \infty} n_x (P(L_k)) =0$. Since
\begin{equation*}
  n_x (L_k) = n_x (P(L_k)) + n_x (N(L_k)),
\end{equation*}
this yields
\begin{equation}\label{eq lim nak}
\lim_{k\rightarrow \infty} n_x (N(L_k)) =1.
\end{equation}
By the rigidity of $N(L_k)$,
\begin{equation*}
N(L_k) = \sum_{i=1} ^ {m_k} a_k ^i [D_k ^i],
\end{equation*}
where each $a_k ^i >0$ and each $D_k ^i$ is an exceptional prime divisor. Note that
\begin{equation*}
  n_x (N(L_k)) = \sum_{i=1} ^ {m_k } a_k ^i \nu([D_k ^i], x),
\end{equation*}
so by (\ref{eq lim nak}), for every $k$ large enough, the above sum has a partial sum, denoted by $\sum_{i=1} ^ {m_k '} a_k ^i \nu([D_k ^i], x)$, so that every $D_k^i$ in this partial sum contains $x$.

By \cite{Bou04}, we know that the number of exceptional primes in every $N(L_k)$ is at most $\rho(X)$, the Picard number of $X$. Thus, $m_k '\leq \rho(X)$ for every $k$. By (\ref{eq lim nak}), there is some $\delta>0$ small such that for all $k$ large enough, there is a term $a_k ^i \nu([D_k ^i], x)$ satisfying
\begin{equation}\label{eq lelong bound}
\frac{1-\delta}{\rho(X)} \leq  a_k ^i \nu([D_k ^i], x) \leq 1+\delta.
\end{equation}

By \eqref{eq lim pos neg} and \eqref{eq lelong bound}, we see that there is a sequence of exceptional prime divisors $D_k$ containing $x$ such that
\begin{equation*}
  \frac{\alpha\cdot D_k}{\nu(D_k, x)} \rightarrow 0.
\end{equation*}

By the definition of Lelong numbers, there is a K\"ahler class $A$ such that $\nu(D_k, x)\leq A^{n-1}\cdot D_k$ holds for all $k$. Consider the sequence $\frac{\{D_k\}}{A^{n-1}\cdot D_k}$ -- it is compact, so after taking a subsequence, we can assume that it is convergent. Moreover,
\begin{equation*}
  \frac{\alpha\cdot D_k}{A^{n-1}\cdot D_k} \rightarrow 0.
\end{equation*}

We claim that the sequence $D_k$ contains a term repeating infinitely many times; hence $D_k / A^{n-1}\cdot D_k$ must contain its limit point. Otherwise, by \cite[Lemma 3.15]{Bou04}, the rays $\mathbb{R}_+ [D_k]$ can accumulate only on $\Mov^1 (X)$, so the limit of the sequence $D_k / A^{n-1}\cdot D_k$ is a non-zero movable class $M$, as each term in the sequence has mass one with respect to $A^{n-1}$. However, since $\alpha \cdot M=0$ and $\mathfrak{M}(\alpha)>0$, this is impossible by \cite{lehmann2016positiivty}. This finishes the proof of the claim. Thus $x$ is contained in some exceptional prime divisor $D$ satisfying $\alpha\cdot D=0$. Then by \cite[Theorem D]{nystrom2016duality}, $D$ must be a divisorial component of $E_{nK}(L_\alpha)$.

In summary, we get
\begin{equation*}
  \mathcal{V}(\alpha) =\ \textrm{the union of divisorial components of}\  E_{nK}(L_\alpha).
\end{equation*}
This finishes the proof of Theorem \ref{thrm local sesha curve}.

\end{proof}

\begin{rmk}
Let $L\in \Eff^1 (X)^\circ$ be a big class and let $\alpha = \langle L^{n-1}\rangle$. By the same argument above, we also have
\begin{equation*}
  \mathcal{V}(\alpha) =\ \textrm{the union of divisorial components of}\  E_{nK}(L).
\end{equation*}
\end{rmk}


\subsection{Further discussions}\label{sec further discuss}

\subsubsection{Local positivity along subvarieties}
We have discussed the local positivity of $(n-1, n-1)$-classes at a point by studying the polar transform of local positivity of $(1,1)$-classes at a point. Note that the local Seshadri and Nakayama constants for $(1,1)$-classes have various generalizations (see e.g. \cite{LazarPositivity-1}, \cite{bauerSeveralauthors-Seshadri}). By the polar transform then, at least at a formal level, these lead to corresponding generalizations for $(n-1, n-1)$-classes.

First, we can study the local positivity along a subvariety:
\begin{itemize}
  \item If $\alpha\in \Eff_1 (X)$, we can define its Nakayama constant along a subvariety $V$ as follows. If $L\in \Nef^1 (X)$, the Seshadri constant of $L$ along $V$ is the real number
      \begin{equation*}
        s_V (L)= \sup\{t\geq 0 | \pi^*L- t E\in \Nef^1 (Y)\},
      \end{equation*}
      where $\pi: Y\rightarrow X$ is the blow-up of $X$ along $V$ and $E$ is the exceptional divisor. The polar transform yields the Nakayama constant of $\alpha$ along $V$:
      \begin{equation*}
        N_V (\alpha)=\inf_{L \in \Nef^1 (X)^\circ} \frac{\alpha\cdot L}{s_V (L)}.
      \end{equation*}
  \item If $\alpha\in \Mov_1 (X)$, we can define its Seshadri constant along a subvariety $V$ as follows. If $L\in \Eff^1 (X)$, the Nakayama constant of $L$ along $V$ is the real number
      \begin{equation*}
        n_V (L)= \sup\{t\geq 0 | \pi^*L- t E\in \Eff^1 (Y)\},
      \end{equation*}
      where $\pi: Y\rightarrow X$ is the blow-up of $X$ along $V$ and $E$ is the exceptional divisor. The polar transform yields the Seshadri constant of $\alpha$ along $V$:
      \begin{equation*}
       S_V (\alpha)=\inf_{L \in \Eff^1 (X)^\circ} \frac{\alpha\cdot L}{n_V (L)}.
      \end{equation*}
\end{itemize}

In the same way, one can also define a higher dimensional Nakayama constant for $(n-1, n-1)$-classes:
\begin{itemize}
  \item Let $L \in \Nef^1 (X)$. Its $d$-dimensional Seshadri constant at $x$ is the real number
      \begin{equation*}
        s^d _x (L)=\inf_{V} \left(\frac{L^d \cdot V}{\nu(V, x)}\right)^{1/d},
      \end{equation*}
      where the infimum is taken over all irreducible subvarieties $V$ of dimension $d$ such that $x\in V$.
      It is not hard to see that the function $s^d _x (\cdot) \in \HConc_1 (\Nef^1 (X))$, so we can define its polar transform. If $\alpha \in \Eff_1 (X)$, its $d$-dimensional Nakayama constant at $x$ is the real number
      \begin{equation*}
      N^d _x (\alpha)=\inf_{L \in \Nef^1 (X)^\circ} \frac{\alpha\cdot L}{s^d _x (L)}.
      \end{equation*}
\end{itemize}

It would be interesting to study the geometry of these invariants. In particular, for the Seshadri function $S_V : \Mov_1 (X) \rightarrow \mathbb{R}$, the generalization of Proposition \ref{prop lowerbound} (i.e., the estimate for $n_V$), would be helpful. One can expect that the mass concentration method developed in \cite{DP04} will give some information on $n_V$.

\begin{rmk}
Let $X$ be a smooth projective surface, and let $V = \{x_1, ..., x_r\}$ be a set of finite points with $r \geq 2$. Then unlike the case when $V$ is a single point (Remark \ref{rmk surface dual}), in general $N_V$ (respectively, $S_V$) does not coincide with $n_V$ (respectively, $s_V$).
\end{rmk}

In \cite{demaillySeshadri}, using singular metrics with isolated singularities, Demailly also introduced another local positivity invariant for a nef line bundle $L$:
\begin{equation*}
  \gamma_x (L)=\sup\{\nu(T, x)|T \in L\  \textrm{is a positive current with isolated singularity at}\ x\}.
\end{equation*}
He also shows that $s_x (L) \geq \gamma_x (L)$, and that if $L$ is ample, then $s_x (L) = \gamma_x (L)$ for every $x\in X$; further, if $L$ is big and nef, then $s_x (L) = \gamma_x (L)$ for any $x$ outside some divisor. Analogous to $\gamma_x (\cdot)$, we can define a similar invariant for a class $\alpha\in \Mov_1 (X)$:
\begin{equation*}
  \Gamma_x (\alpha)=\sup\{\nu(T, x)|T \in \alpha\  \textrm{is a positive current with isolated singularity at}\ x\}.
\end{equation*}
We are not sure if $S_x (\alpha)$ and $\Gamma_x (\alpha)$ have a similar relation as that between $s_x(\cdot)$ and $\gamma_x (\cdot)$.

\subsubsection{Universal generic bounds}
Let $X$ be a smooth projective variety, and let $\alpha$ be a movable curve class. By Remark \ref{rmk nak def}, it is clear that
\begin{equation}\label{eq same sesh}
  S_x (\alpha) = \inf_{x\in D} \frac{\alpha \cdot D}{\mult_x (D)},
\end{equation}
where the infinimum is taken over all irreducible divisors $D$ passing through $x$.

In \cite[Theorem 1]{lazarsfeld_localpositivity} (see also \cite[Section 5.2.C]{LazarPositivity-1}), it is proved that if $L$ is a big and nef line bundle on a projective manifold $X$ of dimension $n$, then
$s_x (L) \geq 1/n$ for all $x$ outside a countable union of subvarieties. Moreover, it is conjectured in \cite[Conjecture 5.2.4]{LazarPositivity-1} that $s_x (L) \geq 1$ for all $x$ outside a countable union of subvarieties. Inspired by this, it is natural to ask:

\begin{quest}
Let $C$ be an irreducible movable curve such that its class $\alpha = \{[C]\}$ is big. Is there a positive constant $c(n)$ depending only on $\dim X =n$ such that
\begin{equation*}
  S_x (\alpha) \geq c(n)
\end{equation*}
for all $x$ outside a countable union of subvarieties?
\end{quest}

By Remark \ref{rmk semicont S_x}, it is clear that if $c$ is the maximum of the function $x\mapsto S_x(\alpha)$, then it assumes $c$ for very general points, that is, outside a countable union of subvarieties.

Analogous to \cite{lazarsfeld_localpositivity}, for the curve class $\alpha=\{[C]\}$, if it satisfies the following condition:
\begin{quote}\label{curve condition}
For any $x$ very general and any irreducible divisor $D$ passing through $x$, there is an effective curve $C_x \in k\alpha$ passing through $x$ such that $C_x \not \subset D$ and
\begin{equation*}
\mult_x (C_x) \geq c(n)k\  \textrm{for some }\ c(n)>0,
\end{equation*}
\end{quote}
then one could get the universal generic lower bound $S_x (\alpha)\geq c(n)$.  Otherwise, suppose the inequality fails -- then for any general point $x$ there is a divisor $D_x$ through $x$ such that $\frac{\alpha \cdot D_x}{\mult_x (D_x)} < c(n)$. However, by the assumption on $\alpha$, there is a curve $C_x \in k\alpha$ such that
\begin{equation*}
k\alpha \cdot D_x \geq \mult_x (C_x) \mult_x (D_x) \geq c(n)k \mult_x (D_x).
\end{equation*}
This is a contradiction.
Note that in the proof for big nef line bundles, the authors used an induction argument to study  subvarieties swept out by ``Seshadri-exceptional curves''. In the curve case, in a similar way, it is also possible to define ``Seshadri-exceptional divisors". But there seem to exist certain difficulties in directly applying the same argument to our setting.

\begin{rmk}
In \cite{fulger2017seshadri}, Fulger noticed the following estimate: let $L \in \Nef^1 (X)$ and let $\alpha = L^{n-1}$. Then $S_x (\alpha)\geq s_x (L)^{n-1}$. This is true by a direct intersection number calculation:
\begin{equation*}
  (L- s_x (L) E)^{n-1} = \alpha + s_x (L)^{n-1} e.
\end{equation*}
Hence,
if $\alpha = L^{n-1}$ for some big nef line bundle $L$, then by \cite{lazarsfeld_localpositivity} one has the following generic lower bound
\begin{equation*}
S_x (\alpha) \geq  \frac{1}{n^{n-1}}.
\end{equation*}
\end{rmk}

\begin{rmk}
When $X$ is a smooth projective surface and $\alpha$ is given by an ample line bundle, it is known that $S_x (\alpha) = s_x (\alpha) \geq 1$ for all $x$ outside perhaps countably many points (see \cite[Proposition 5.2.3]{LazarPositivity-1}).
\end{rmk}

\bibliography{reference}
\bibliographystyle{amsalpha}

\bigskip

\bigskip

\noindent
\textsc{Department of Mathematics, Northwestern University,
Evanston, IL 60208, USA}\\
\noindent
\verb"Email: njm2@math.northwestern.edu"\\
\verb"Email: jianxiao@math.northwestern.edu"
\end{document}